\documentclass[graybox]{svmult}

\usepackage{helvet}         
\usepackage{courier}        
\usepackage{makeidx}         
\usepackage{graphicx}        
\usepackage{multicol}        
\usepackage[bottom]{footmisc}
\usepackage{amssymb,amsmath,url}

\newcommand{\rank}{\hbox{rank}}
\newcommand{\norm}[1]{\left\Vert#1\right\Vert}

\makeatletter
\renewcommand*\env@matrix[1][c]{\hskip -\arraycolsep
  \let\@ifnextchar\new@ifnextchar
  \array{*\c@MaxMatrixCols #1}}
\makeatother

\smartqed

\allowdisplaybreaks

\usepackage{float}
\usepackage{enumitem}
\usepackage[hypertexnames=false,colorlinks=true,breaklinks=true,bookmarks=true,urlcolor=blue,citecolor=blue,linkcolor=blue,bookmarksopen=false,draft=false]{hyperref}

\usepackage[capitalise]{cleveref}

\begin{document}

\title*{1-norm minimization and minimum-rank structured sparsity
for symmetric and ah-symmetric  generalized inverses: rank one and two}
\titlerunning{Approximate 1-norm minimization and minimum-rank structured sparsity}
\author{Luze Xu, Marcia Fampa, Jon Lee}
\institute{
Luze Xu \at IOE Dept., Univ. of Michigan, Ann Arbor, MI, USA. \email{xuluze@umich.edu}\\[3pt]
Marcia Fampa \at Universidade Federal do Rio de Janeiro, Brazil. \email{fampa@cos.ufrj.br}\\[3pt]
Jon Lee \at IOE Dept., Univ. of Michigan, Ann Arbor, MI, USA. \email{jonxlee@umich.edu}
}

\date{\today}

\maketitle

\abstract{
Generalized inverses are important in statistics and other areas of applied matrix algebra. 
A \emph{generalized inverse} of a real matrix $A$
is a matrix $H$  that satisfies the Moore-Penrose (M-P) property
$AHA=A$. If $H$ also satisfies the  M-P property $HAH=H$, then it is
called \emph{reflexive}. Reflexivity of a generalized inverse is equivalent to minimum rank, a highly desirable property.
We consider aspects of symmetry related to
the calculation of various \emph{sparse} reflexive generalized inverses of $A$.
As is common, we use (vector) 1-norm minimization for both inducing sparsity and for keeping the magnitude
of entries under control.
\newline
\phantom{a}\quad  When $A$ is symmetric,  a symmetric $H$ is highly desirable, 
but generally such a restriction
on $H$ will not lead to a 1-norm minimizing reflexive generalized inverse.
We investigate a block construction method to produce a symmetric reflexive generalized inverse that is structured and has guaranteed sparsity.
Letting  the rank of $A$ be $r$, we establish that the 1-norm minimizing 
generalized inverse of this type is a
1-norm minimizing symmetric generalized inverse when (i) $r=1$ and when (ii) $r=2$ and $A$ is nonnegative. 
\newline
\phantom{a}\quad
Another aspect of symmetry that we consider relates to another M-P property:
$H$ is \emph{ah-symmetric} if $AH$ is symmetric. The ah-symmetry property
is sufficient for a generalized inverse to be used to 
 solve the least-squares problem $\min\{\|Ax-b\|_2:~x\in\mathbb{R}^n\}$ using $H$,
 via $x:=Hb$.
We investigate a column
block construction method to produce an ah-symmetric  reflexive generalized inverse
that is structured and has guaranteed sparsity.
We establish that the 1-norm minimizing  ah-symmetric
generalized inverse of this type is a 1-norm minimizing ah-symmetric generalized inverse when (i) $r=1$ and
when (ii) $r=2$ and $A$ satisfies a technical condition.
}

\section{Introduction}\label{sec: intro}

Generalized inverses are essential tools in statistics and 
other applications of matrix algebra.
Of central importance is the celebrated Moore-Penrose (M-P) pseudoinverse,
which can be used, for example, to calculate the least-squares solution of an over-determined system of linear equations.
With respect to  the system  $Ax=b$,  the least-squares solution $\min\{\|Ax-b\|_2:~x\in\mathbb{R}^n\}$
is given by $x:=A^+b$, where $A^+$ is the M-P pseudoinverse.
Our motivating \emph{use case} is that we have a very large (rank deficient) matrix $A$
and multiple right-hand sides $b$. And so we can  see the value of having a sparse
generalized inverse.

In what follows,  we write $\|H\|_1$ to mean $\|\mathrm{vec}(H)\|_1$ and $\|H\|_{\max}$ to mean $\|\mathrm{vec}(H)\|_{\max}$ (in both cases, these are not the usual induced/operator matrix norms). We use $I$ for an identity matrix and $J$ for an all-ones matrix. Matrix dot product is indicated by $\langle X, Y\rangle:=\sum_{ij}x_{ij}y_{ij}=\mathrm{trace}(X^\top Y)$. We use $A[S,T]$ for the submatrix of $A$ with row indices $S$ and column indices $T$; additionally, we use $A[S,:]$ ( resp., $A[:,T]$) for the submatrix of $A$ formed by the rows $S$ (resp., columns $T$).
Finally, if $A$ is symmetric, we let $A[S]:=A[S,S]$,
the principal submatrix of $A$ with  row/column indices $S$.

When a real matrix $A\in\mathbb{R}^{m\times n}$ is not square or is square but not invertible, we consider ``pseudoinverses'' of $A$ (see \cite{rao1971}). The most well-known pseudoinverse is the \emph{M-P pseudoinverse}
independently discovered by A. Bjerhammar, E.H. Moore and R. Penrose
(see \cite{Bjerhammar,Moore,Penrose}). If $A=U\Sigma V^\top$ is the real singular-value decomposition of $A$ (see \cite{GVL1996}, for example), where $U\in\mathbb{R}^{m\times m}$, $V\in\mathbb{R}^{n\times n}$ are orthogonal matrices and $\Sigma=\mathrm{diag}(\sigma_1,\sigma_2,\dots,\sigma_p)\in\mathbb{R}^{m\times n}$ ($p=\min\{m,n\}$) with  singular values $\sigma_1\ge\sigma_2\ge\dots\ge\sigma_p\ge0$, then the M-P pseudoinverse of $A$ can be defined as $A^+:=V\Sigma^+U^\top$, where $\Sigma^+:=\mathrm{diag}(\sigma_1^+,\sigma_2^+,\dots,\sigma_p^+)\in\mathbb{R}^{n\times m}$, $\sigma_i^+:=1/\sigma_i$ for all $\sigma_i\ne 0$, and $\sigma_i^+:=0$ for all $\sigma_i=0$. 
The starting point for our investigation is the following celebrated result. 

\begin{theorem}[see \cite{Penrose}]
For $A \in \mathbb{R}^{m \times n}$, the M-P pseudoinverse $A^+$ is the unique $H \in \mathbb{R}^{n \times m}$ satisfying:
	\begin{align}
		& AHA = A \label{property1} \tag{P1}\\
		& HAH = H \label{property2} \tag{P2}\\
		& (AH)^{\top} = AH \label{property3} \tag{P3}\\
		& (HA)^{\top} = HA \label{property4} \tag{P4}
	\end{align}
\end{theorem}

Following \cite{RohdeThesis}, a \emph{generalized inverse} is any $H$ satisfying  \ref{property1}. 
Because we are interested in sparse $H$, \ref{property1} is important to enforce, otherwise the completely sparse zero-matrix
(which carries no information from $A$) always satisfies
the other three M-P properties. 
A generalized inverse is \emph{reflexive} if it satisfies \ref{property2}.  Two very useful facts
are: (i) if $H$ is a generalized inverse of $A$, then $\mathrm{rank}(H)\ge\mathrm{rank}(A)$, and (ii) a generalized inverse $H$ of $A$ is reflexive if and only if $\mathrm{rank}(H)=\mathrm{rank}(A)$
(see  \cite[Theorem 3.14]{RohdeThesis}). 
Therefore, enforcing \ref{property2} for a generalized inverse
implies that the generalized inverse has minimum rank.
A low-rank $H$ can be viewed as being more interpretable/explainable model 
(say in the context of
the least-squares problem), so we naturally prefer reflexive generalized inverses (which have the least rank possible among generalized inverses).
As we have said, we are interested in
sparse generalized inverses. But structured sparsity of $H$ is even more valuable, as
it can be viewed, in a different way, as being a more interpretable/explainable model.

Following \cite{XFLP}, if $H$ satisfies \ref{property3}, we say that $H$ is \emph{ah-symmetric}.
That is, ah-symmetric  
means that $AH$ is symmetric.
It is very important to know that not  all of the M-P properties are required for a generalized inverse to exactly solve key problems. For example, if $H$ is any ah-symmetric generalized inverse, then $\hat{x}:=Hb$ solves the least-squares problem $\min\{\|Ax-b\|_2:~x\in\mathbb{R}^n\}$ (see \cite{campbell2009generalized} or \cite{FFL2016}, for example).

{\bf Previous work:}
\cite{dokmanic,dokmanic1,dokmanic2} used sparse-optimization techniques to give tractable right and left sparse pseudoinverses.
\cite{FFL2016} introduced the idea of sparse generalized inverses obtained by 1-norm minimization over convex
relaxations of the M-P properties (also see \cite{FFL2019}).
\cite{FampaLee2018ORL} introduced ``block soultions'' giving reflexive generalized inverses
with structured sparsity. They solved the 1-norm minimization problem for the rank-1 case and the rank-2 nonnegative case,
and they gave an efficient approximation algorithm for the general-rank case. 
For symmetric reflexive generalized inverses, \cite{XFLP}  gave a 1-norm 
approximation algorithm for the general-rank case, based on a local search over symmetric block solutions.
Additionally, \cite{XFLP} introduced ``column block solutions'' 
giving reflexive ah-symmetric generalized inverses
with structured sparsity. Furthermore, 
they gave an efficient 1-norm approximation algorithm for the general-rank case,  
based on a local search over column block solutions.
\cite{FLPX} makes a detailed computational study of all these approximation 
algorithms (also see \cite{XFLP}).

The approximation algorithms aimed at minimizing  $\|H\|_1$ from \cite{FampaLee2018ORL} and \cite{XFLP}
(over $H$ satisfying \ref{property1}+\ref{property2}, or \ref{property1}+\ref{property2}+\,$H$\!=\!$H^\top$, or \ref{property1}+\ref{property2}+\ref{property3})
carry out local searches (using appropriate local-search neighborhoods, over, respectively, block solutions, symmetric block solutions, or column block solutions --- all reflexive and structured), 
but not with $\|H\|_1$ as a 
local-minimization criterion. 
In fact, in  \cite{XFLP} it was demonstrated that
such a criterion can fail to provide a decent approximation ratio for $r:=\rank(A)\ge 2$. 
 However, the global minimizer with respect to the criterion $\|H\|_1$ might 
 provide a better approximation ratio than the local search.
If $r$ is very small, we could calculate a globally-minimum $\|H\|_1$ (over block solutions, symmetric block solutions, or column block solutions, respectively) and hope that such a  solution 
could be a 1-norm minimizing solution (satisfying \ref{property1}+\ref{property2}, or \ref{property1}+\ref{property2}+\,$H$\!=\!$H^\top$, or \ref{property1}+\ref{property2}+\ref{property3}, respectively).
In fact, \cite{FampaLee2018ORL} demonstrated this for the case of \ref{property1}+\ref{property2},
when $r=1$ and when $r=2$ for nonnegative $A$.

{\bf Our main contributions:} In \S\ref{sec:sym},  (i) we solve the 1-norm minimization problem 
for \ref{property1}+\ref{property2}+\,$H$\!=\!$H^\top$ (i.e.,  symmetric reflexive generalized inverses), in 
the rank-1 case and the rank-2 nonnegative case.
In \S\ref{sec:ah-sym}, 
(ii) we solve the 1-norm minimization problem 
for \ref{property1}+\ref{property2}+\ref{property3} (i.e., 
ah-symmetric reflexive generalized inverses), in 
the rank-1 case and the rank-2 case under
an efficiently-checkable technical condition, and 
(iii) we demonstrate that this technical condition is
essentially necessary. 

Already, the rank-2 cases need side conditions and have rather complicated proofs. 
So some new insights will be needed to characterize optimality beyond
the cases that we handle. Nevertheless, we believe that our results can
be stepping stones to characterizing optimal solutions or getting
better approximation ratios in further low-rank cases. 

In what follows, a key tool that we employ is linear-optimization duality.
Even though we are interested in \emph{reflexive} generalized inverses,
it is useful to consider 
 $\min\{\|H\|_1: \ref{property1}\}=\min\{\norm{H}_1: AHA=A\}$, which we re-cast as a linear-optimization problem \eqref{eqn:P} and its dual \eqref{eqn:D}:
\begin{equation*}\label{eqn:P}\tag{P}
\begin{array}{ll}
\mbox{minimize }& \langle J, H^+\rangle + \langle J, H^-\rangle\\
\mbox{subject to} & A(H^+ - H^-)A=A,\\
& H^+,H^-\ge 0;
\end{array}
\end{equation*}
\begin{equation*}\label{eqn:D}\tag{D}
\begin{array}{ll}
\mbox{maximize }& \langle A, W\rangle\\
\mbox{subject to} & -J\le A^\top W A^\top\le J.\\
\end{array}
\end{equation*}
More compactly, we can recast \eqref{eqn:D} as: $\max\{\langle A,W\rangle:~\|A^\top WA^\top\|_{\max}\le 1\}$.

\section{Symmetric}\label{sec:sym}

Our starting point is the construction of \emph{symmetric block solutions} via the following result.

\begin{theorem}[\cite{XFLP}]\label{thm:symconstruction}
For a symmetric matrix $A\in\mathbb{R}^{n\times n}$, let $r := \rank(A)$. Let $\tilde{A}:=A[S]$ be any $r \times r$
nonsingular principal submatrix of $A$. Let $H\in\mathbb{R}^{n\times n}$ be equal to zero, except its submatrix with row/column indices $S$ is equal to $\tilde{A}^{-1}$. Then $H$ is a symmetric reflexive generalized inverse of $A$.
\end{theorem}

\subsection{Rank 1}
Next, we demonstrate that when $\rank(A) = 1$, construction
of a $1$-norm minimizing symmetric reflexive generalized inverse can be based on the symmetric block construction over the diagonal elements of $A$.
\begin{theorem}
Let $A$ be an arbitrary rank-1 symmetric matrix, which is, without loss of generality, of the form $A:=uu^\top$, where $\mathbf{0}\ne u\in\mathbb{R}^n$. If $i^*:=\arg\max_{i}\{|u_i|\}=\arg\max_{i}\{|a_{ii}|\}$, then $H:=\frac{1}{u_{i^*}^2}e_{i^*} e_{i^*}^\top$,
where $e_{i^*}\in\mathbb{R}^n$ is a standard unit vector, is a symmetric reflexive generalized inverse of $A$ with minimum 1-norm.
\end{theorem}
\begin{proof}
We consider \eqref{eqn:P} and \eqref{eqn:D}.
A feasible solution for \eqref{eqn:P} is $H^+=\frac{1}{u_{i^*}^2}e_{i^*} e_{i^*}^\top$, $H^-=\mathbf{0}$.
A feasible solution for \eqref{eqn:D} is $W=\frac{1}{u_{i^*}^4}e_{i^*} e_{i^*}^\top$~, because $\norm{A^\top W A^\top}_{\max} =\frac{1}{u_{i^*}^2}\norm{A}_{\max}=1$.
And the objective value of the dual solution is $\langle A,W\rangle=u_{i^*}^2\cdot \frac{1}{u_{i^*}^4} = 1/u_{i^*}^2$, which is the objective value of the primal solution. Therefore, by the weak-duality theorem of linear optimization, we have that $H:=H^+-H^-$ is a generalized inverse of $A$ with minimum $1$-norm. By our construction, $H$ is symmetric and reflexive. Therefore, $H$ is a symmetric reflexive generalized inverse with minimum 1-norm. \qed
\end{proof}
\smallskip

Another way to view the rank-1 case is by using the Kronecker product to transform the constraint $AHA=A$ into $[A^\top \otimes A]\mathrm{vec}(H)=\mathrm{vec}(A)$. Note that $\mathrm{vec}(A)=\mathrm{vec}(u u^\top)=u\otimes u$, and $A^\top \otimes A= u u^\top \otimes u u^\top = [u\otimes u][u^\top \otimes u^\top]$. So the constraint becomes
$
[u\otimes u]\left([u^\top \otimes u^\top]\mathrm{vec}(H)\right)=u\otimes u~\Leftrightarrow~ [u \otimes u]^\top \mathrm{vec}(H)=1.
$
Thus the $1$-norm minimization may be re-cast as $\min\{\norm{\mathrm{vec}(H)}_1:~[u \otimes u]^\top \mathrm{vec}(H)=1\}$, or
$\min\{\norm{H}_1:~u^\top Hu=1\}$, or $\min\{\norm{H}_1:~\langle uu^\top , H\rangle =1\}$. By using the inequality $x^\top y\le \norm{x}_{\infty}\norm{y}_1$ or $\langle X, Y\rangle\le \norm{X}_{\max}\norm{Y}_1$, we have $\norm{H}_1\ge 1/\norm{u u^\top}_{\max}$, and the equality holds when $H=\frac{1}{u_{i^*}^2}e_{i^*} e_{i^*}^\top$~.

\subsection{Rank 2}
Generally, when $\rank(A)=2$, we cannot construct a $1$-norm minimizing symmetric reflexive generalized inverse based on the symmetric block construction.
For example, with
$$A:=\begin{bmatrix}[r]
5& 4& 2\\
4& 5& -2\\
2& -2& 8
\end{bmatrix},$$
we have a symmetric reflexive generalized inverse
$H:=\frac{1}{81}A\quad (\text{because}~A^2=9A)$,
with $\|H\|_1=\frac{34}{81}$. While the three symmetric reflexive generalized inverses based on the symmetric block construction have $1$-norm equal to $\frac{17}{36},\frac{17}{36},2$, all greater than $\frac{34}{81}$.

Next, we demonstrate that under the natural but restrictive condition that $A$ is non-negative, when $\rank(A) = 2$, construction of a $1$-norm minimizing symmetric reflexive generalized inverse can be based on the symmetric block construction over the $2\times 2$ principal submatrix of $A$.

\begin{theorem}
\label{thmr2}
Let $A$ be an arbitrary rank-$2$ non-negative symmetric matrix. For any $i_1,i_2\in \{1,\ldots,n\}$, with $i_1< i_2$, let $\tilde{A}:=A[\{i_1,i_2\}]$. If $i_1,i_2$ are chosen to minimize the 1-norm of $\tilde{A}^{-1}$ among all nonsingular $2\times 2$ principal submatrices, then the $n\times n$ matrix $H$ constructed by \cref{thm:symconstruction} over $\tilde{A}$, is a symmetric reflexive generalized inverse of $A$ with minimum $1$-norm.
\end{theorem}
\begin{proof}
Without loss of generality that $\tilde{A}$ is in the north-west corner of $A$. So we take $A$ to have the form $\begin{bmatrix}\tilde{A} & B\\ B^\top & D\end{bmatrix}$. Let $M=2I-J$ if $\det(\tilde{A})>0$ and $M=J-2I$ if $\det(\tilde{A})<0$, 
now we let
$$
W:=\begin{bmatrix}
\tilde{W} & 0\\
0 & 0
\end{bmatrix}
:=\begin{bmatrix}
\tilde{A}^{-\top}M\tilde{A}^{-\top} & 0\\
0 & 0
\end{bmatrix}.
$$
The dual objective value
$$\langle A,W\rangle=\mathrm{trace}(A^\top W)=\mathrm{trace}(\tilde{A}^\top \tilde{W})=\mathrm{trace}(M\tilde{A}^{-\top})=\langle M,\tilde{A}^{-1}\rangle=\norm{\tilde{A}^{-1}}_1,$$
i.e., $\langle A,W\rangle = \norm{H}_1$.
Also,
$$
A^\top W A^\top=\begin{bmatrix}
M & M\tilde{A}^{-\top}B\\
B^\top\tilde{A}^{-\top}M & B^\top\tilde{A}^{-\top}M\tilde{A}^{-\top}B
\end{bmatrix}.
$$
Clearly $\norm{M}_{\max}\le 1$. Next, we consider $\bar{\gamma}:=M\tilde{A}^{-\top}\gamma=M\tilde{A}^{-1}\gamma$, where $\gamma$ is an arbitrary column of $B$. As $\mathrm{rank}(A)=2$ and $\tilde{A}$ is nonsingular, we assume that $\gamma = \tilde{A}\begin{bmatrix}x_1\\x_2\end{bmatrix}$. We may as well assume that $x_1,x_2$ are not both zero, because otherwise $\bar{\gamma}=0$ satisfying $\|\bar{\gamma}\|_{\max}\le 1$. We have
$$
\bar{\gamma} = M\begin{bmatrix}x_1\\x_2\end{bmatrix} \quad\Rightarrow\quad \|\bar{\gamma}\|_{\max} = |x_1-x_2|.
$$
Consider $\tilde{A}$ is chosen to minimize the 1-norm of $\tilde{A}^{-1}$ among all nonsingular $2\times 2$ principal submatrices, we have
\begin{align}
&\|\tilde{A}^{-1}\|_1\le \left\|\begin{bmatrix}a_{11} & a_{11}x_1+a_{12}x_2\\
a_{11}x_1+a_{12}x_2 & a_{11}x_1^2+2a_{12}x_1x_2+a_{22}x_2^2\end{bmatrix}^{-1}\right\| \label{eqn:1}\\
&\|\tilde{A}^{-1}\|_1\le \left\|\begin{bmatrix} a_{22} & a_{12}x_1+a_{22}x_2\\
a_{12}x_1+a_{22}x_2 & a_{11}x_1^2+2a_{12}x_1x_2+a_{22}x_2^2\end{bmatrix}^{-1}\right\| \label{eqn:2}
\end{align}
\begin{enumerate}[label=\textbf{Case \arabic*.}]
    \item{$x_1=0$.} If $a_{12}=0$, then $a_{11},a_{22}>0$ because $\tilde{A}$ is nonsingular. Using \eqref{eqn:1}, we have
    $$
    \frac{a_{11}+2a_{12}+a_{22}}{|\det(\tilde{A})|} \le \frac{a_{11}(1+x_1)^2+2a_{12}x_2(1+x_1)+a_{22}x_2^2}{x_2^2|\det(\tilde{A})|}
    $$
    Simplifying, we obtain
    $
    a_{11}(x_2^2-1)\le 0 ~\Rightarrow~ |x_2|\le 1.
    $
    If $a_{12}>0$, then $x_2\ge 0$ because $a_{11}x_1+a_{12}x_2=a_{12}x_2\ge0$. Because $x_1$,$x_2$ are not both zero, we have $x_2>0$. Still using \eqref{eqn:1}, we have
    $
    (a_{11}x_2+a_{11}+2a_{12}x_2)(x_2-1)\le 0
    $
    which implies $x_2\le 1$.
    \item{$x_2=0$.} Similarly by using \eqref{eqn:2}, we have $|x_1|\le 1$ or $0<x_1\le 1$.
    \item{$x_2\ge x_1$, $x_1,x_2\ne 0$.} Using \eqref{eqn:1}, we have
    $$
    \frac{a_{11}+2a_{12}+a_{22}}{|\det(\tilde{A})|} \le \frac{a_{11}(1+x_1)^2+2a_{12}x_2(1+x_1)+a_{22}x_2^2}{x_2^2|\det(\tilde{A})|}
    $$
    Simplifying, we obtain
    $
    (a_{11}(x_2+1+x_1) + 2a_{12}x_2)(x_2-1-x_1)\le 0
    $
    Because $a_{11}(x_2+1+x_1) + 2a_{12}x_2 = 2(a_{11}x_1+a_{12}x_2) + a_{11}(1+x_2-x_1)>0$, (it is zero only when $a_{11}=0$ and $a_{11}x_1+a_{12}x_2=0$, which implies $a_{11}=a_{12}=0$, a contradiction.) we obtain $0\le x_2-x_1\le 1$.
    \item{$x_2< x_1$, $x_1,x_2\ne 0$.} Using \eqref{eqn:2}, we have
    $$
    \frac{a_{11}+2a_{12}+a_{22}}{|\det(\tilde{A})|} \le \frac{a_{11}x_1^2+2a_{12}x_1(1+x_2)+a_{22}(1+x_2)^2}{x_1^2|\det(\tilde{A})|}
    $$
    Simplifying, we obtain
    $
    (a_{22}(x_1+1+x_2) + 2a_{12}x_1)(x_1-1-x_2)\le 0
    $
    Because $a_{22}(x_1+1+x_2) + 2a_{12}x_1 = 2(a_{12}x_1+a_{22}x_2) + a_{22}(1-x_2+x_1)>0$, (it is zero only when $a_{22}=0$ and $a_{12}x_1+a_{22}x_2=0$, which implies $a_{22}=a_{12}=0$, a contradiction.) we obtain $0< x_1-x_2\le 1$.
\end{enumerate}
From the above, we show that $|x_1-x_2|\le 1$, thus $\|M\tilde{A}^{-\top}B\|_{\max} = \|B^\top \tilde{A}^{-\top}M\|_{\max}\le 1$. Finally,
\begin{align*}
\|B^\top\tilde{A}^{-\top}M\tilde{A}^{-\top}B\|_{\max}&=\frac{1}{2}\|B^\top\tilde{A}^{-\top}(4I-2J)\tilde{A}^{-\top}B\|_{\max}\\
&=\frac{1}{2}\|B^\top\tilde{A}^{-\top}M^2\tilde{A}^{-\top}B\|_{\max}\\
&\le \|B^\top \tilde{A}^{-\top}M\|_{\max}\|M\tilde{A}^{-\top}B\|_{\max}\le 1
\end{align*}
Therefore, $W$ is dual feasible. By the weak duality, we know that $H$ is a symmetric generalized inverse of $A$ with minimum $1$-norm, which is also reflexive.
\qed
\end{proof}


\section{ah-symmetric}\label{sec:ah-sym}

Our starting point is the construction of \emph{column block solutions} via the following result.

\begin{theorem}[\cite{XFLP}]\label{thm:ahconstruction}
For $A\in\mathbb{R}^{m\times n}$, let $r := \rank(A)$. For any $T$, an ordered subset of $r$ elements from $\{1,\dots,n\}$, let $\hat{A}:=A[:,T]$ be the $m \times r$ submatrix of $A$ formed by columns $T$. If $\rank(\hat{A})=r$, let
$
\hat{H} := \hat{A}^+ =(\hat{A}^\top\hat{A})^{-1}\hat{A}^\top.
$
The $n \times m$ matrix $H$ with all rows equal to zero, except rows $T$, which are given by $\hat{H}$, is an ah-symmetric reflexive generalized inverse of $A$.
\end{theorem}

Similarly as before, we note that it is useful to consider   $\min\{\|H\|_1: \ref{property1}+\ref{property3}\}=\min\{\norm{H}_1: AHA=A,~(AH)^\top=AH\}$, which we re-cast as a linear-optimization problem \eqref{eqn:Pah} and its dual \eqref{eqn:Dah}:
\begin{equation*}\label{eqn:Pah}\tag{\text{$P_{ah}$}}
\begin{array}{ll}
\mbox{minimize }& \langle J, H^+\rangle + \langle J, H^-\rangle\\
\mbox{subject to} & A(H^+ - H^-)A=A,\\
& (H^+ - H^-)^\top A^\top = A(H^+ - H^-),\\
& H^+,H^-\ge 0.
\end{array}
\end{equation*}
\begin{equation*}\label{eqn:Dah}\tag{\text{$D_{ah}$}}
\begin{array}{ll}
\mbox{maximize }& \langle A, W\rangle\\
\mbox{subject to} & -J\le A^\top W A^\top + A^\top (V^\top-V)\le J\\
\end{array}
\end{equation*}
We can see \eqref{eqn:Dah}   as: $\max\{\langle A,W\rangle:~\|A^\top WA^\top+A^\top U\|_{\max}\le 1, ~U^\top=-U\}$.

\subsection{Rank 1}
Next, we demonstrate that when $\rank(A) = 1$, construction
of a $1$-norm minimizing ah-symmetric reflexive generalized inverse can be based on the column block construction.
\begin{theorem}
\label{thmr1withP3b}
Let $A$ be an arbitrary $m\times n$, rank-1 matrix.
For any $j\in\{1,\dots,n\}$, let $\hat{a}$ be column $j$ of $A$. If $j$ is chosen to minimize the 1-norm of $\hat{a}^+$ among all columns except the zero columns, then the $n\times m$ matrix $H$ constructed by \cref{thm:ahconstruction} over $\hat{a}$, is an ah-symmetric reflexive generalized inverse of $A$ with minimum 1-norm.
\end{theorem}

\begin{proof}
We prove a stronger result --- that our constructed $H$ is a 1-norm minimizing ah-symmetric generalized inverse. By our construction, $H$ is reflexive, thus $H$ is an ah-symmetric reflexive generalized inverse with minimum 1-norm. To establish the minimum 1-norm of $H$, we consider the linear-optimization problems \eqref{eqn:Pah} and \eqref{eqn:Dah}. As verified in \cref{thm:ahconstruction}, $H$ is a feasible solution for \eqref{eqn:Pah}, and its objective value is
$
\|H\|_1= \|\hat{a}^+\|_1
$
(it also satisfies the nonlinear equations \eqref{property2}).

The objective function of \eqref{eqn:Dah} only depends on the variable $W$.
Feasibility of a $W$ is equivalent to the existence of
a skew-symmetric matrix $U$ so that
\begin{equation}
\label{dconst0}
\| A^\top W A^\top +  A^\top U\|_{\max} \leq 1~.
\end{equation}
Next, we are going to construct a dual feasible solution $W$ with objective value $\langle A,W\rangle = \|H\|_1$~; then by the weak duality for linear optimization, we establish that $H$ is optimal to \eqref{eqn:Pah}.

Let $z:=\mbox{sign}(\hat{a}^+)$. Suppose that $\hat{a}_i$ is a nonzero element in $\hat{a}$ with index $i$. Let $W$ be an $m\times n$ matrix with all elements equal to zero, except the one in row $i$ and column $j$, which is given by $\hat{w}$. Let $U$ be an $m\times m$ skew-symmetric matrix, with only row $i$ and column $i$ different from zero.

If $\hat{w}$ and $U$ are chosen to be
$$
\hat{w} := \frac{1}{\hat{a}_i} z(\hat{a}^+)^\top,~
u_{ki}=-u_{ik}:= \frac{1}{\hat{a}_i}(
\hat{a}_kz (\hat{a}^+)^\top-z_k),~\forall ~k\ne i,
$$
then they satisfy
\begin{equation}
\label{econd0}
\hat{a}_{i}\hat{w}\hat{a}^\top + \hat{a}^\top U=z~.
\end{equation}
This is because for $k\ne i$, $\hat{a}_i\hat{w}\hat{a}_k+\hat{a}_i u_{ik}=z(\hat{a}^+)^\top\hat{a}_k+z_k-\hat{a}_kz(\hat{a}^+)^\top=z_k$, and
\begin{align*}
 &\hat{a}_i\hat{w}\hat{a}_k+\sum_{k\ne i}\hat{a}_k u_{ki}=z(\hat{a}^+)^\top(\hat{a}_i+\sum_{k\ne i}\frac{\hat{a}_k^2}{\hat{a}_i})-\sum_{k\ne i}\frac{\hat{a}_k}{\hat{a}_i}z_k\\
 &\qquad =\frac{1}{a_i}\Big(z(\hat{a}^+)^\top(\hat{a}^\top\hat{a})-z\hat{a}\Big)+z_i=z_i~,   
\end{align*}
and
$$
\mbox{trace}(A^\top W)=\hat{a}_{i}\hat{w} = z(\hat{a}^+)^\top =\|\hat{a}^+\|_1=\|H\|_1~.
$$

The dual constraint \eqref{dconst0} can be written as
\begin{equation}
\label{dconsta0}
\| \hat{a}^\top W A^\top +  \hat{a}^\top U\|_{\max} \leq 1~,\,
\end{equation}
and
\begin{equation}
\label{dconstb0}
\| \hat{B}^\top W A^\top +  \hat{B}^\top U\|_{\max} \leq 1~.\,
\end{equation}
From \eqref{econd0}, we have that \eqref{dconsta0} is satisfied.
To verify \eqref{dconstb0}, let  $\hat{b}\in\mathbb{R}^m$ be an arbitrary column of $\hat{B}$.
As $A$  has rank 1,
\[
\hat{b} = \alpha \hat{a}~,
\]
and
\[
\hat{b}^\top W A^\top +  \hat{b}^\top U = \alpha \hat{a}^\top(W A^\top + U) ~.
\]
Considering \eqref{econd0}, we have that $\| \hat{a}^\top(W A^\top + U)\|_{\max}  = 1$, and therefore,
\begin{equation*}
\label{rel1}
\|\hat{b}^\top W A^\top +  \hat{b}^\top U\|_{\max} =  |\alpha| ~.
\end{equation*}
We have
\[
\hat{a}^+= (\hat{a}^\top\hat{a})^{-1}\hat{a}^\top=\frac{1}{\hat{a}^\top\hat{a}}\hat{a}^\top~.
\]
We also have
\[
\hat{b}^+= (\hat{b}^\top\hat{b})^{-1}\hat{b}^\top
=\frac{1}{\hat{b}^\top\hat{b}}\hat{b}^\top~
=\frac{1}{\alpha (\hat{a}^\top\hat{a})}\hat{a}^\top~.
\]
From optimality of $H$, we have
$
\|\hat{a}^+\|_1\leq \|\hat{b}^+\|_1~.
$
Therefore
\[
\frac{1}{|\hat{a}^\top\hat{a}|}\|\hat{a}\|_1
\leq \frac{1}{|\alpha|\, |\hat{a}^\top\hat{a}|}\|\hat{a}\|_1~.
\]
So, $ |\alpha| \leq 1 $, which implies that \eqref{dconstb0} is satisfied.
\qed
\end{proof}
\smallskip
Before moving on to the rank-$2$ case, we generalize 
the choice of $\hat{w},U$ satisfying \eqref{econd0} to the general rank-$r$ case.
\begin{theorem}[\cite{XFLP}]\label{thm:WU}
Let $T$ be an ordered subset of $r$ elements from $\{1,\dots,n\}$ and $\hat{A}:=A[:,T]$ be the $m\times r$ submatrix of an $m\times n$ matrix $A$ formed by columns $T$, and $\rank(\hat{A})=r$. There exists an $m\times n$ matrix $W$ and a skew-symmetric $m\times m$ matrix $U$ such that
\begin{equation*}
\hat{A}^\top W A^\top + \hat{A}^\top U=Z,
\end{equation*}
where $Z :=\mbox{sign}(\hat{A}^+)$. Furthermore, $\langle A,W\rangle=\|\hat{A}^+\|_1$.
\end{theorem}


\subsection{Rank 2}
Generally, when $\rank(A)=2$, we cannot construct a $1$-norm minimizing ah-symmetric reflexive generalized inverse based on the column block construction. Even under the condition that $A$ is non-negative, we have the following example:
$$
A =\begin{bmatrix}1&3&8\\ 2&2&8\\ 3&1&8\end{bmatrix}.
$$
Note that $\mathrm{rank}(A)=2$ because $a_3=2a_1+2a_2$. We have an ah-symmetric reflexive generalize inverse with $1$-norm $\frac98$,
$$
H :=\begin{bmatrix}[r]
-\frac14 & 0 &\frac14 \\
\frac14 & 0 & -\frac14\\
\frac{1}{24} &\frac{1}{24} & \frac{1}{24}
\end{bmatrix}.
$$
However, the three ah-symmetric reflexive generalized inverses based on our column block construction have $1$-norm
$\frac{31}{24}, \frac{31}{24}, \frac{7}{6}$, respectively.

Next, we demonstrate that under an efficiently-checkable technical condition, when $\rank(A) = 2$, construction of a $1$-norm minimizing ah-symmetric reflexive generalized inverse can be based on the column block construction.

\begin{theorem}
\label{thmr2withP3}
Let $A$ be an arbitrary $m\times n$, rank-$2$ matrix.
For any $j_1,j_2\in \{1,\ldots,n\}$, with $j_1< j_2$, let $\hat{A}:=[\hat{a}_{j_1},\hat{a}_{j_2}]$ be the $m\times 2$ submatrix of $A$ formed by columns  $j_1$ and $j_2$.
Suppose that  $j_1,j_2$ are chosen to minimize the 1-norm of $\hat{H}:=\hat{A}^+$ among all $m\times 2$ rank-$2$ submatrices of $A$. Every column $\hat{b}$ of $A$, can be uniquely written in the basis $\hat{a}_{j_1},\hat{a}_{j_2}$, say
$\hat{b}=\alpha \hat{a}_{j_1} +\beta\hat{a}_{j_2}$.
Suppose that for each such column $\hat{b}$ of $A$,
one of the following conditions holds on the associated $\alpha$, $\beta$:
\vskip5pt
\begin{enumerate}[label=(\roman*),leftmargin=2\parindent,itemsep=1ex]
\item \label{p1} $|\alpha|+|\beta|\le 1;$
\item \label{p2} $\hat{H}_{1j}\hat{H}_{2j}\le 0$ for $j=1,\dots,m$, and $\alpha\beta\ge 0$;
\item \label{p3} $\hat{H}_{1j}\hat{H}_{2j}\ge 0$ for $j=1,\dots,m$, and $\alpha\beta\le 0$.
\end{enumerate}
Then the $n\times m$ matrix $H$ constructed by \cref{thm:ahconstruction} based on $\hat{A}$, is an ah-symmetric reflexive generalized inverse of $A$ with minimum 1-norm.
\end{theorem}


\begin{proof}


We prove a stronger result that our constructed $H$ is a 1-norm minimizing ah-symmetric generalized inverse. By our construction, $H$ is reflexive, thus $H$ is an ah-symmetric reflexive generalized inverse with minimum 1-norm. To establish the minimum 1-norm of our constructed $H$, we consider the dual pair of  linear-optimization problems \eqref{eqn:Pah} and \eqref{eqn:Dah}. As verified in \cref{thm:ahconstruction}, $H$ is a feasible solution for \eqref{eqn:Pah}, and its objective value is
$
\|H\|_1= \|\hat{A}^+\|_1
$
(it also satisfies the nonlinear equations \eqref{property2}).

The objective-function of \eqref{eqn:Dah} only depends on the variable $W$.
Feasibility of $W$ is equivalent to the existence of
a skew-symmetric matrix $U$ satisfying
\begin{equation}
\label{dconst}
\| A^\top W A^\top +  A^\top U\|_{\max} \leq 1~.
\end{equation}
Next, we are going to construct a dual feasible solution $W$ with objective value $\langle A,W\rangle = \|H\|_1$, then by the weak duality for linear optimization, we prove that $H$ is optimal to \eqref{eqn:Pah}.

By \cref{thm:WU}, we can choose $W$ and a skew-symmetric matrix $U$ such that
\begin{equation}
\label{econd}
\hat{A}^\top W A^\top + \hat{A}^\top U=Z~,
\end{equation}
and then
$$
\langle A,W\rangle = \|\hat{A}^+\|_1=\|H\|_1~.
$$

The dual constraint \eqref{dconst} can be written as
\begin{equation}
\label{dconsta}
\| \hat{A}^\top W A^\top +  \hat{A}^\top U\|_{\max} \leq 1~,\,
\end{equation}
and
\begin{equation}
\label{dconstb}
\| \hat{B}^\top W A^\top +  \hat{B}^\top U\|_{\max} \leq 1~.\,
\end{equation}
From \eqref{econd}, we have that \eqref{dconsta} is satisfied.
To verify \eqref{dconstb}, without loss of generality, let $(j_1,j_2)=(1,2)$, and let $\hat{b}\in\mathbb{R}^m$ be an
 arbitrary column of $\hat{B}$ with $\alpha$ and $\beta$ such that $\hat{b}=\alpha\hat{a}_1+\beta\hat{a}_2$~; thus
\[
\hat{b}^\top W A^\top +  \hat{b}^\top U = [\alpha ~\beta](\hat{A}^\top W A^\top +\hat{A}^\top U) = [\alpha ~\beta]Z=[\alpha ~\beta]\mbox{sign}(\hat{H}).
\]
\begin{itemize}[leftmargin=1\parindent]
\item For case \ref{p1}, we have $\|\hat{b}^\top W A^\top +  \hat{b}^\top T\|_{\max}\le (|\alpha| +|\beta|)\|Z\|_{\max}\le 1$.
\item For case \ref{p2}, because $\hat{H}_{1j}\hat{H}_{2j}\le 0$ for $j=1,\dots,m$, we have $\hat{b}^\top W A^\top +  \hat{b}^\top U = (\alpha-\beta)\mbox{sign}(\hat{H}_{1\cdot})$, and thus
$
\|\hat{b}^\top W A^\top +  \hat{b}^\top U\|_{\max}=|\alpha-\beta|.
$
Also we have $\alpha\beta\ge 0$, so
$
\|\hat{b}^\top W A^\top +  \hat{b}^\top U\|_{\max}=\big||\alpha|-|\beta|\big|.
$
\item For case \ref{p3}, because $\hat{H}_{1j}\hat{H}_{2j}\ge 0$ for $j=1,\dots,m$,, we have $\hat{b}^\top W A^\top +  \hat{b}^\top U = (\alpha+\beta)\mbox{sign}(\hat{H}_{1\cdot})$, and thus
$
\|\hat{b}^\top W A^\top +  \hat{b}^\top U\|_{\max}=|\alpha+\beta|.
$
Also we have $\alpha\beta\le 0$~; so
$
\|\hat{b}^\top W A^\top +  \hat{b}^\top U\|_{\max}=\big||\alpha|-|\beta|\big|.
$
\end{itemize}

So to prove the dual feasibility, we only need to show that $||\alpha| - |\beta||\leq1$ .

Let $\hat{A}_{\hat{b}/1}:=[\hat{b}\;\;\hat{a}_2]$ and $\delta_{ij}:=\hat{a}^\top_i\hat{a}_j$, for $i,j=1,2$. We have
\begin{align*}
\hat{A}^+ &= (\hat{A}^\top\hat{A})^{-1}\hat{A}^\top=
 \left(
\left[\begin{array}{c}\hat{a}^\top_1\\ \hat{a}^\top_2\end{array}\right]
[\hat{a}_1\;\;\hat{a}_2]
\right)^{-1}
\left[\begin{array}{c}\hat{a}^\top_1\\ \hat{a}^\top_2\end{array}\right]\\
&=\frac{1}{\theta} \left[\begin{array}{rr}\delta_{22}&-\delta_{12}\\ -\delta_{12}&\delta_{11}\end{array}\right]\left[\begin{array}{c}\hat{a}^\top_1\\ \hat{a}^\top_2\end{array}\right]\\
&=\frac{1}{\theta} \left[\begin{array}{r}\delta_{22}\hat{a}^\top_1-\delta_{12}\hat{a}^\top_2\\ -\delta_{12}\hat{a}^\top_1+\delta_{11}\hat{a}^\top_2\end{array}\right]~,
\end{align*}
where $ \theta = \delta_{11}\delta_{22} - \delta_{12}^2$. We also have
\begin{align*}
\hat{A}_{\hat{b}/1}^+&= (\hat{A}_{\hat{b}/1}^\top\hat{A}_{\hat{b}/1})^{-1}\hat{A}_{\hat{b}/1}^\top\\&=
 \left(
\left[\begin{array}{c}\alpha \hat{a}^\top_{1} + \beta \hat{a}^\top_{2}\\ \hat{a}^\top_2\end{array}\right]
[\alpha \hat{a}_{1} + \beta \hat{a}_{2}\;\;\hat{a}_2]
\right)^{-1}
\left[\begin{array}{c}\alpha \hat{a}^\top_{1} + \beta \hat{a}^\top_{2}\\ \hat{a}^\top_2\end{array}\right]\\
&=\frac{1}{\tilde{\theta}} \left[\begin{array}{cc}\delta_{22}&-\alpha \delta_{12} - \beta \delta_{22}\\ -\alpha \delta_{12} - \beta \delta_{22}&
\alpha^2 \delta_{11} + 2\alpha \beta \delta_{12} + \beta^2 \delta_{22}
\end{array}\right]\left[\begin{array}{c}\alpha\hat{a}^\top_1 + \beta \hat{a}^\top_2\\ \hat{a}^\top_2\end{array}\right]\\
&= \frac{1}{\tilde{\theta}} \left[\begin{array}{cc}\alpha \delta_{22}\hat{a}^\top_1 + \beta \delta_{22}\hat{a}^\top_2-\alpha \delta_{12}\hat{a}^\top_2-\beta \delta_{22}\hat{a}^\top_2\\
-(\alpha \delta_{12} +  \beta
\delta_{22})(\alpha\hat{a}^\top_1 + \beta\hat{a}^\top_2)
+(\alpha^2 \delta_{11}  + 2\alpha \beta \delta_{12} + \beta^2 \delta_{22})\hat{a}^\top_2
\end{array}\right]\\
&= \frac{1}{\tilde{\theta}} \left[\begin{array}{cc}\alpha \delta_{22}\hat{a}^\top_1 -\alpha \delta_{12}\hat{a}^\top_2\\
-\alpha^2 \delta_{12}\hat{a}^\top_1- \alpha \beta
\delta_{22}\hat{a}^\top_1
+\alpha^2 \delta_{11} \hat{a}^\top_2 + \alpha \beta \delta_{12}\hat{a}^\top_2
\end{array}\right]\\
&= \frac{\alpha}{\tilde{\theta}} \left[\begin{array}{cc}\delta_{22}\hat{a}^\top_1 - \delta_{12}\hat{a}^\top_2\\
-\alpha \delta_{12}\hat{a}^\top_1-  \beta
\delta_{22}\hat{a}^\top_1
+\alpha \delta_{11} \hat{a}^\top_2 +  \beta \delta_{12}\hat{a}^\top_2
\end{array}\right]\\
&= \frac{\alpha}{\tilde{\theta}} \left[\begin{array}{cc}\delta_{22}\hat{a}^\top_1 - \delta_{12}\hat{a}^\top_2\\
-\alpha ( \delta_{12}\hat{a}^\top_1- \delta_{11} \hat{a}^\top_2) -  \beta (
\delta_{22}\hat{a}^\top_1 - \delta_{12}\hat{a}^\top_2)
\end{array}\right]~,\\
\end{align*}
where
\begin{align*}
\tilde{\theta}&=  \delta_{22}(\alpha^2 \delta_{11} + 2\alpha \beta \delta_{12} + \beta^2 \delta_{22})  - (\alpha \delta_{12} + \beta \delta_{22})^2\\
&=  \alpha^2 \delta_{11}\delta_{22} + 2\alpha \beta \delta_{12}\delta_{22} + \beta^2 \delta_{22}^2  - (\alpha^2 \delta_{12}^2 + 2\alpha  \beta \delta_{12} \delta_{22}+  \beta^2 \delta_{22}^2)\\
&=  \alpha^2 (\delta_{11}\delta_{22}  -  \delta_{12}^2 )
 =\alpha^2 \theta~.
\end{align*}
From optimality of $H$, we have
$
\|\hat{A}^+\|_1\leq \|\hat{A}_{\hat{b}/1}^+\|_1~.
$
Therefore
\begin{align*}
&\frac{1}{|\theta|} \left(\|\delta_{22}\hat{a}^\top_1-\delta_{12}\hat{a}^\top_2\|_1 + \| -\delta_{12}\hat{a}^\top_1+\delta_{11}\hat{a}^\top_2\|_1\right)  \\
\leq & \frac{|\alpha|}{|\tilde{\theta}|} \left(\|\delta_{22}\hat{a}^\top_1 - \delta_{12}\hat{a}^\top_2\|_1 + \|-\alpha ( \delta_{12}\hat{a}^\top_1- \delta_{11} \hat{a}^\top_2) -  \beta (
\delta_{22}\hat{a}^\top_1 - \delta_{12}\hat{a}^\top_2)\|_1\right)\\
 = & \frac{1}{|\alpha\theta|} \left(\|\delta_{22}\hat{a}^\top_1 - \delta_{12}\hat{a}^\top_2\|_1 + \| \alpha( -\delta_{12}\hat{a}^\top_1+ \delta_{11} \hat{a}^\top_2) +  \beta (
-\delta_{22}\hat{a}^\top_1 + \delta_{12}\hat{a}^\top_2)\|_1\right)\\
\leq &  \frac{1}{|\alpha\theta|} \left(\|\delta_{22}\hat{a}^\top_1 - \delta_{12}\hat{a}^\top_2\|_1 +  \| \alpha( -\delta_{12}\hat{a}^\top_1+ \delta_{11} \hat{a}^\top_2)\|_1 + \|\beta( -
\delta_{22}\hat{a}^\top_1 + \delta_{12}\hat{a}^\top_2)\|_1\right)~\\
 = & \frac{1}{|\alpha\theta|} \left(\|\delta_{22}\hat{a}^\top_1 - \delta_{12}\hat{a}^\top_2\|_1 +  |\alpha| \, \|  -\delta_{12}\hat{a}^\top_1+ \delta_{11} \hat{a}^\top_2\|_1 + |\beta| \, \| -
\delta_{22}\hat{a}^\top_1 + \delta_{12}\hat{a}^\top_2\|_1\right).
\end{align*}
So,
\begin{align*}
& |\alpha| \left(\|\delta_{22}\hat{a}^\top_1-\delta_{12}\hat{a}^\top_2\|_1 + \| -\delta_{12}\hat{a}^\top_1+\delta_{11}\hat{a}^\top_2\|_1\right)  \\
\leq & \|\delta_{22}\hat{a}^\top_1 - \delta_{12}\hat{a}^\top_2\|_1 + |\alpha| \, \|  -\delta_{12}\hat{a}^\top_1+ \delta_{11} \hat{a}^\top_2\|_1 + |\beta| \, \| -
\delta_{22}\hat{a}^\top_1 + \delta_{12}\hat{a}^\top_2\|_1, 
\end{align*}
and
\begin{equation}
\label{rell1}
 |\alpha| -  |\beta| \leq 1.
\end{equation}

Now, considering that
$
\|\hat{A}^+\|_1\leq \|\hat{A}_{\hat{b}/2}^+\|_1~,
$
where $\hat{A}_{\hat{b}/2}:=[\hat{a}_1 \;\; \hat{b}]$, we  obtain
\begin{align*}
&\frac{1}{|\theta|} \left(\|\delta_{22}\hat{a}^\top_1-\delta_{12}\hat{a}^\top_2\|_1 + \| -\delta_{12}\hat{a}^\top_1+\delta_{11}\hat{a}^\top_2\|_1\right)  \\
&  \leq  \frac{|\beta|}{|\tilde{\theta}|} \left(\|-\delta_{12}\hat{a}^\top_1 + \delta_{11}\hat{a}^\top_2\|_1 + \|\alpha ( \delta_{12}\hat{a}^\top_1- \delta_{11} \hat{a}^\top_2) +  \beta (
\delta_{22}\hat{a}^\top_1 - \delta_{12}\hat{a}^\top_2)\|_1\right)\\
&  \leq  \frac{1}{|\beta\tilde{\theta}|} \left(\|-\delta_{12}\hat{a}^\top_1 + \delta_{11}\hat{a}^\top_2\|_1 + |\alpha| \, \|  \delta_{12}\hat{a}^\top_1- \delta_{11} \hat{a}^\top_2\|_1 +  |\beta| \, \|   \delta_{22}\hat{a}^\top_1 - \delta_{12}\hat{a}^\top_2\|_1\right)~.
\end{align*}
So,
\begin{align*}
& |\beta| \left(\|\delta_{22}\hat{a}^\top_1-\delta_{12}\hat{a}^\top_2\|_1 + \| -\delta_{12}\hat{a}^\top_1+\delta_{11}\hat{a}^\top_2\|_1\right)  \\
& \leq  \|-\delta_{12}\hat{a}^\top_1 + \delta_{11}\hat{a}^\top_2\|_1 + |\alpha| \, \|  \delta_{12}\hat{a}^\top_1- \delta_{11} \hat{a}^\top_2\|_1 +  |\beta| \, \|   \delta_{22}\hat{a}^\top_1 - \delta_{12}\hat{a}^\top_2\|_1~, 
\end{align*}
and
\begin{equation}
\label{rell2}
 |\beta| -  |\alpha|
 \leq  1  \\~.
\end{equation}

From \eqref{rell1} and \eqref{rell2}, we have
$
 \big||\alpha| -  |\beta|\big|
 \leq  1~.
$
\qed
\end{proof}
\smallskip
\begin{remark}
The following example shows that we can allow different cases for each column $\hat{b}$. Let
\[
A := \begin{bmatrix}
2  &   3  &   1  &   5\\
2  &   3  &   1  &   5\\
2  &   5  &   2  &   7
\end{bmatrix}.
\]
Then $\{i_1,i_2\}=\{1,2\}$ minimizes the $1$-norm of $\hat{H}$ with $\|\hat{H}\|=3$. We have
\[
\hat{H}=
\begin{bmatrix}[r]
\frac{5}{8}  &   \frac58  &   -\frac34\\[3pt]
-\frac14  &   -\frac14  &   \frac12
\end{bmatrix}
\]
satisfying $\hat{H}_{1j}\hat{H}_{2j}\le 0$ for $j=1,2,3$, and
$\hat{b}_1= [1,1,2]^\top = -\frac14 \hat{a}_1 +\frac12\hat{a}_2$ satisfies only case \ref{p1}, and $\hat{b}_2= [5,5,7]^\top = \hat{a}_1 +\hat{a}_2$ satisfies only case \ref{p2}.
\end{remark}

\smallskip
The technical sufficient condition in \cref{thmr2withP3}, while efficiently checkable, may
seem rather complicated.
But perhaps surprisingly, if $H$ having minimum 1-norm of $\hat{H}$ is an optimal solution to \eqref{eqn:Pah} (i.e., a 1-norm minimizing ah-symmetric generalized inverse of $A$, \emph{following our column block construction}), then the condition is also necessary. So, for rank-$2$, there
is no possibility of further generalizing the condition, in the context of
proving the optimality of our chosen column block construction.

\begin{theorem}\label{thmr2withP3necessity}
Let $A$ be an arbitrary $m\times n$, rank-$2$ matrix.
For any $j_1,j_2\in \{1,\ldots,n\}$, with $j_1< j_2$, let $\hat{A}:=[\hat{a}_{j_1},\hat{a}_{j_2}]$ be the $m\times 2$ submatrix of $A$ formed by columns  $j_1$ and $j_2$.
Suppose that  $j_1,j_2$ are chosen to minimize the 1-norm of $\hat{H}:=\hat{A}^+$ among all $m\times 2$ rank-$2$ submatrices of $A$.
Suppose that the $n\times m$ matrix $H$ constructed by \cref{thm:ahconstruction} based on $\hat{A}$, is an ah-symmetric generalized inverse of $A$ with minimum 1-norm.
 Every column $\hat{b}$ of $A$, can be uniquely written in the basis $\hat{a}_{j_1},\hat{a}_{j_2}$, say
$\hat{b}=\alpha \hat{a}_{j_1} +\beta\hat{a}_{j_2}$.
Then for each such column $\hat{b}$ of $A$,
one of the following conditions holds on the associated $\alpha$, $\beta$:
\smallskip
\begin{enumerate}[label=(\roman*),leftmargin=2\parindent,itemsep=1ex]
\item $|\alpha|+|\beta|\le 1;$
\item $\hat{H}_{1j}\hat{H}_{2j}\le 0$ for $j=1,\dots,m$, and $\alpha\beta\ge 0$;
\item $\hat{H}_{1j}\hat{H}_{2j}\ge 0$ for $j=1,\dots,m$, and $\alpha\beta\le 0$.
\end{enumerate}
\end{theorem}
\begin{proof}
We consider the dual pair of linear-optimization problems \eqref{eqn:Pah} and \eqref{eqn:Dah}. Because $H$ is an optimal solution to \eqref{eqn:Pah}, by the complementary slackness, we have
\begin{align*}
    &\langle J - A^\top W A ^\top - A^\top U , H^+ \rangle= 0,\\
    &\langle J + A^\top W A ^\top + A^\top U , H^- \rangle= 0,
\end{align*}
where $H^+=\max\{H,0\}$, $H^-=-\min\{H,0\}$, $W,U$ is an optimal solution to \eqref{eqn:Dah} with $U^\top=-U$. Along with $H^+,H^-\ge0$ and dual feasiblity, we have
\begin{align*}
    &(J - A^\top W A ^\top - A^\top U)_{ij} H_{ij}^+=0,\\
    &(J + A^\top W A ^\top + A^\top U)_{ij} H_{ij}^-= 0.
\end{align*}
Thus,
$$
(A^\top WA^\top +A^\top U)_{ij} = \begin{cases}
    1, & H_{ij} > 0,\\
    -1, & H_{ij} < 0,\\
    [-1,1], & H_{ij} = 0.
\end{cases}
$$
Without loss of generality, assume that $(j_1,j_2)=(1,2)$, and $H = [\hat{H};~0]$ with $\hat{H}\in\mathbb{R}^{2\times m}$. Let
$$(A^\top WA^\top +A^\top U)[\{1,2\},:] = \hat{A}^\top W A^\top + \hat{A}^\top U := Z.$$
For every column $\hat{b}$ of $A$, $\hat{b}=\alpha \hat{a}_{1} +\beta\hat{a}_{2}$ because $\mathrm{rank}(A) = 2$. Hence
$$\hat{b}^\top W A^\top + \hat{b}^\top U = [\alpha~\beta](\hat{A}^\top W A^\top + \hat{A}^\top U) = [\alpha~\beta]Z,$$
and we have $\|\hat{b}^\top W A^\top + \hat{b}^\top U\|_{\max}\le1$.
\begin{itemize}[leftmargin=1\parindent]
	\item If for any $j\in\{1,\dots,m\}$, one of $\hat{H}_{1j},\hat{H}_{2j}$ is zero, then $\hat{H}_{1j}\hat{H}_{2j}= 0$ for $j=1,\dots,m$, thus for any $\alpha,\beta$, either $\alpha\beta\ge0$ or $\alpha\beta\le 0$ holds.
    \item If $\hat{H}_{1j}\hat{H}_{2j}\le 0$ for $j=1,\dots,m$, and $\hat{H}_{1k}\hat{H}_{2k}< 0$ for some $k$, then $Z_{1k}=-Z_{2k}\ne0$, and $|\alpha-\beta|=|\alpha Z_{1k}+\beta Z_{2k}|\le 1$, thus $|\alpha|+|\beta|\le 1$ or $\alpha\beta\ge0$.
    \item If $\hat{H}_{1j}\hat{H}_{2j}\ge 0$ for $j=1,\dots,m$, and $\hat{H}_{1k}\hat{H}_{2k}> 0$ for some $k$, then $Z_{1k}=Z_{2k}\ne0$, and $|\alpha+\beta|=|\alpha Z_{1k}+\beta Z_{2k}|\le 1$, thus $|\alpha|+|\beta|\le 1$ or $\alpha\beta\le0$.
    \item Otherwise, we have both $|\alpha-\beta|\le 1$ and $|\alpha+\beta|\le 1$, which implies $|\alpha|+|\beta|\le 1$.
\end{itemize}
Hence $\alpha,\beta$ must satisfy one of \ref{p1}, \ref{p2} and \ref{p3}.
\qed
\end{proof}

\section{Conclusion and open questions}
When $A$ is symmetric, a 1-norm minimizing symmetric block solution is a 1-norm minimizing symmetric generalized inverse when (i) r = 1 and when (ii) r=2 and $A$ is nonnegative. A 1-norm minimizing column block solution is a 1-norm minimizing ah-symmetric generalized inverse when (i) r = 1 and when (ii) r=2 and $A$ satisfies a technical condition.
It would be interesting to investigate the approximation ratio achieved by a 1-norm 
minimizing symmetric block solution and column block solution for general $r$. 
Also characterizing optimality conditions beyond the cases that we studied is a nice challenge.

\begin{acknowledgement}
M. Fampa was supported in part by CNPq grant 303898/2016-0.
J. Lee was supported in part by ONR grant N00014-17-1-2296.
\end{acknowledgement}

\bibliographystyle{alpha}

\end{document}